\documentclass{article}
\usepackage{amsmath}
\usepackage{amssymb}
\usepackage{eucal}
\usepackage{cite}
\usepackage{hyperref}

\newtheorem{theorem}{Theorem}[section]

\newtheorem{question}[theorem]{Question}

\newtheorem{corollary}[theorem]{Corollary}

\newtheorem{definition}[theorem]{Definition}
\newtheorem{example}[theorem]{Example}

\newtheorem{lemma}[theorem]{Lemma}

\newenvironment{proof}[1][Proof]{\textbf{#1 }}

\begin{document}

 \baselineskip 15pt

\title{Groups in which some primary subgroups are weakly s-supplemented\thanks{The
 author was supported by  NNSF of P. R. China (Grant 11071229),
    Scientific Research Foundation of CUIT(Grant J201114) and
  a China Postdoctoral Science Foundation funded project  (Grant  20110491726)}}
\author{Baojun Li\\
{\small College of Applied Mathematics,
Chengdu University of Information Technology }\\
{\small Chengdu   610225, P. R. China}\\
{\small\  E-mail: baojunli@cuit.edu.cn}\\
}

\date{}
\maketitle

\abstract
A subgroup $H$ of a group $G$ is called
weakly s-supplemented in $G$ if there is a subgroup $T$ such that $G=HT$ and $H\cap T\le H_{sG}$,
where $H_{sG}$ is  the subgroup of $H$
generated by all those subgroups of $H$ which are s-permutable in $G$. The influence of primary
 weakly s-supplemented subgroups on the
structure of finite groups is investigated.  An open question promoted by Skiba
is studied and some known results are generalized.

\section{Introduction}
All groups  considered in this paper are finite. The notions and notations not introduced are standard and
the reader is referred to \cite{D-H-BOOK,Huppert-BOOK-I,GUO-BOOK}
if necessary.

A subgroup $H$ of a group $G$ is said to be a permutable
subgroup (cf. \cite{D-H-BOOK}) of $G$ or a quasinormal subgroup of $G$ (cf. \cite{Ore-qusinormal}) if
$H$ is permutable with all subgroups of $G$.  The permutability of subgroups plays an
  important role in the study of the structure of finite groups and was
 generalized extensively. Recall that
 a subgroup $H$  of a group $G$ is called s-permutable (or s-quasinormal) in $G$
if $H$ permutes with every Sylow subgroup of $G$(cf. \cite{Kegel}). Let $H$ be a subgroup of $G$. $H_{sG}$ denotes  the subgroup of $H$ generated by all
those subgroups of $H$ which are s-permutable in $G$. In \cite{On-weakly-s-permutable-subgroups}, the following definitions are introduced.

 \begin{definition} {\rm\cite{On-weakly-s-permutable-subgroups}}
 Let $H$ be a subgroup  $G$. $H$ is called
weakly s-supplemented in $G$ if there is a subgroup $T$ such that $G=HT$ and $H\cap T\le H_{sG}$, and if $T$ is subnormal in $G$ then $H$ is called
weakly s-permutable in $G$
\end{definition}

By using this idea, Skiba
 \cite{On-weakly-s-permutable-subgroups} proved the following nice result.
  \begin{theorem}
    Let
$\frak{F}$ be a saturated formation containing all supersoluble
groups and $G$ a group with a normal subgroup $E$ such that
 $G/E \in\frak{F}$. Suppose that every
non-cyclic Sylow subgroup $P$ of $E$ has a subgroup $D$ such that $1
< |D| < |P|$ and all subgroups $H$ of $P$ with order $|H| = |D|$ and
with order $2|D|$ (if $P$ is a non-abelian 2-group and $|P : D|> 2$)
not having a supersoluble supplement in $G$ are weakly s-permutable
in $G$. Then  $G \in \frak{F}$.
  \end{theorem}

 The above theorem generalized many known results. in connection with this, the following question was proposed by A. Skiba.

\begin{question}
 {\rm\cite[Question 6.4]{On-weakly-s-permutable-subgroups}}
 Let $\frak{F}$ be a saturated formation containing all supersolvable groups and $G$ a group
with a normal subgroup $E$ such that $G/E \in \frak{F}$. Suppose that every non-cyclic Sylow subgroup $P$
of $E$ has a subgroup $D$ such that $1 < |D| < |P|$ and all subgroups $H$ of $P$ with order $|H| = |D|$
and with order $2|D|$ (if $P$ is a non-abelian 2-group and $|P : D| > 2$) are weakly s-supplemented
in G. Is then $G \in \frak{F}$?
\end{question}

We have given an example in \cite{Guo-Xie-Li} to show that the answer of this  question is negative in general.
But, in the following theorem, we will prove  in many case the question has positive answer.

For convenience, if $m=p^\alpha$ is a $p$-number,  let  $\iota(m)$ denote log$_pm=\alpha$ and if $P$ is a $p$ group we use $\iota(P)$
instead  of $\iota(|P|)$.

\begin{theorem}\label{B}

Let $\frak{F}$ be a saturated formation containing all supersolvable groups and $G$ a group
with a normal subgroup $E$ such that $G/E \in \frak{F}$. Suppose that every non-cyclic Sylow subgroup
$P$ of $E$ has a subgroup $D$ such that $1 < |D| < |P|$ and all subgroups $H$ of $P$ with order
$|H| = |D|$ and with order $2|D|$ (if $P$ is a non-abelian 2-group and $|P : D|> 2$ ) having no supersolvable supplement in $G$ are weakly
s-supplemented in $G$. If one of the following holds:\\
$(i)$ $\Phi(P)\ne P'$;\\
$(ii)$ $|D|\le |P'|$;\\
$(iii)$ $|P'|<|D|$ and  $(\iota(P/P'),\iota(|D|/|P'|))=1$ or $(\iota(P),\iota(|P:D|))=1$;\\
then $G \in \frak{F}$
\end{theorem}

\section{Elementary Properties}

 \begin{lemma}{\rm(\cite[Lemma 2.10]{On-weakly-s-permutable-subgroups})}\label{sk1} Let $G$ be a group and $H \le K \le G$. Then

\noindent$(i)$  Suppose that $H$ is normal in $G$. Then $K/H$ is weakly s-supplemented in $G/H$ if and only if
$K$ is weakly  s-supplemented in $G$.\\
$(ii)$ If $H$ is weakly s--supplemented in $G$, then $H$ is weakly s-supplemented in $K$.\\
$(iii)$ Suppose that $H$ is normal in $G$. Then the  $HE/H$ is weakly
s-supplemented in $G/H$ for every weakly  weakly s-supplemented
in $G$ subgroup $E$ satisfying $(|H|, |E|) = 1$.
 \end{lemma}

\begin{lemma}{\rm{(\cite[Lemma 2.8]{open-problems-related-X-permutability})}}\label{2.7}
Let $\frak{F}$ be a saturated formation and $P$ be a normal
$p$-subgroup of $G$. Then $P\subseteq Z_\infty^\frak{F}(G)$ if and
only if $P/\Phi(P)\subseteq Z_\infty^\frak{F}(G/\Phi (P))$.
 \end{lemma}

\begin{lemma}{\rm(\cite[Lemma 2.7]{On-weakly-s-permutable-subgroups})} \label{spn}
    If $H$ is s-permutable in a group $G$ and $H$ is a $p$-group for some prime $p$, then
$O^p(G) \le N_G(H)$.
\end{lemma}

\begin{lemma}{\rm(\cite[Lemma3.8.7]{GUO-BOOK}} \label{OP}
    Let $G$ be a $p$-solvable group. If $O_{p'}(G)=1$ and $O_p(G)\le H\le G$, then $O_{p'}(H)=1$
\end{lemma}

The following Lemma is well known.

\begin{lemma}\label{lesoc}
    Let $N$ be a nilpotent normal subgroup of $G$. If $N\cap \Phi(G)=1$ then $N\le $Soc$(G)$, that is,
    $N=N_1\times N_2\times\cdots\times N_r$, where
    $N_1, N_2,\cdots, N_r$ are minimal normal subgroups of $G$.
\end{lemma}

\begin{lemma}{\rm(\cite{Prime-power-index})} \label{prime-index}
    Let $G$ be a nonabelian simple group and $H$ a subgroup of $G$. If $|G:H|=p^a$, where $p$ is a prime. Then one of the following holds:\\
    $(i)$ $G=A_n$, $H\cong A_{n-1}$, where $n=p^a$;\\
    $(ii)$  $G=PSL_n(q),\ |G:H|=(q^n-1)/(q-1)=p^a$;\\
    $(iii)$  $G=PSL_n(11),\ H\cong A_5$;\\
    $(iv)$  $G=M_{23}$ and $H\cong M_{22}$ or $G=M_{11}$ and $H\cong M_{10}$;\\
    $(v)$  $G=PSU_4(2)$,the index of $H$ in $G$ is 27.

\end{lemma}

\begin{lemma}\label{max}
    Let $G$ be a group, $p$ the minimal prime divisor of the order of $G$ and $P$ a Sylow $p$-subgroup of $G$. If every maximal subgroup of $P$ having no supersolvable supplement in $G$ is weakly s-supplemented in $G$ then $G$ is $p$-nilpotent.
\end{lemma}
\begin{proof}
It can be obtained directly from \cite[Theorem C]{Guo-Xie-Li}.

   \begin{lemma}{\rm
   (\cite[III, 5.2 and IV, 5.4]{Huppert-BOOK-I}) } \label{non-p}
   Suppose that $p$ is a prime and $G$ is a
minimal non-$p$-nilpotent group. Then\\
$(i)$ $G$ has a normal Sylow $p$-subgroup $P$   and $G = PQ$,
where $Q$ is a non-normal cyclic $q$-subgroup for some prime $q\ne p$.\\
$(ii)$ $P/\Phi(P)$ is a minimal normal subgroup of $G/\Phi(P)$.\\
$(iii)$ If $P$ is abelian or $p > 2$, then exp$(P) = p$.\\
$(iv)$ If P is non-abelian and $p = 2$, then exp$(P) = 4$.
   \end{lemma}

\end{proof}

\begin{lemma}\label{minimal}
Let $G$ be a group, $p$ the minimal prime divisor of the order of $G$ and $P$ a Sylow $p$-subgroup of $G$. If every  subgroup  of $P$ of order
$p$ or 4(when $P$ is a non abelian 2-group) having no supersolvable supplement in $G$ is weakly s-supplemented in $G$ then $G$ is $p$-nilpotent.
\end{lemma}
\begin{proof}
  By Lemma \ref{sk1}, one can verify that the hypotheses are subgroups closed. Thus if $G$ is not $p$-nilpotent then we can assume that
  $G$ is a a minimal non-$p$-nilpotent group,  and hence, by Lemma \ref{non-p},
    $G=P\rtimes Q$, where $Q$ is a cyclic $q$-group for some prime $q$,
    $P/\Phi(P)$ is a chief factor of $G$ and exp$P=p$ or 4 (when $P$ is noncyclic 2-group). Let $a\in P\setminus\Phi(P)$ and $H=\langle a\rangle$.
    Then $|H|=p$ or 4 (when $P$ is a nonabelian 2-group). Thus $H$ either has a supersolvable supplement in $G$ or is weakly s-supplemented in $G$.
    Assume $|H|=2$. If $H$
    has a complement $T$ in $G$, then $T$ is $p$-nilpotent since $G$ is minimal non-$p$-nilpotent group.
   Since $|G:T|=|HT:T|=|H|=p$,  $T\unlhd G$. This induces that $G$ is $p$-nilpotent, a contradiction.
   Thus $G$
  is the only supplement of $H$ in $G$. If $H$ has a
  supersolvable supplement in $G$, then $G$ is supersolvable and so is $p$-nilpotent since
   $p$ is minimal. If $H$ is weakly s-supplemented in $G$ then $H$ is s-permutable in $G$ and hence $H\Phi(P)/\Phi(P)$ is s-permutable in $G/\Phi(P)$.
 It follows from Lemma \ref{spn} that $H\Phi(P)/\Phi(P)\unlhd G/\Phi(P)$ and so $P/\Phi(P)=H\Phi(P)$ is cyclic of order $p$. Hence $|G:Q\Phi(P)|=p$
and so $Q\phi(P)\unlhd G$. This implies that $G/\Phi(P)$ is cyclic and so is $G$, a contradiction.
If $|H|=4$, considering the subgroup $H\Phi(P)/\Phi(P)$ in $G/\Phi(P)$, a contradiction can also
be obtained by a similar argument. Therefore the lemma holds.
\end{proof}

\section{Proof of Theorem \ref{B}}
\begin{lemma}\label{31}
    Let $P$ be a normal $p$-subgroup of $G$ with $P\cap \Phi(G)=1$. Assume that $D$ is a subgroup of $P$ with $1<D<P$. If every subgroup of order $|D|$
    of $P$ having no supersolvable supplement in $G$ is weakly s-supplement in $G$, then $P=P_1\times P_2\times\cdots\times P_r$, where $P_1,  P_2, \cdots, P_r$ are all minimal normal in $G$ of same order and  $\iota(D)=m\iota(P_i)$, that is, $|D|=|P_i|^m$  for some positive integer $m$, $i=1,2, \cdots,\ r$.
\end{lemma}
\begin{proof}
    By Lemma \ref{lesoc}, $P=P_1\times P_2\times\cdots\times P_r$, where $P_1,  P_2, \cdots, P_r$ are all minimal normal in $G$. Assume that there is a $P_i$
    such that $|D|<|P_i|$. Then $P_i$ has a proper subgroup $H$ of order $|D|$. Moreover, by the property of $p$-groups we can choose $H$ to be  normal in some Sylow $p$-subgroup $G_p$
     of $G$ containing $P$. By the
hypotheses, $H$ either has a supersolvable supplement in $G$ or
is weakly s-supplement in $G$. Let $T$ be a supplement of $H$
in $G$. Then $G=HT=P_iT$ and $P_i\cap T\ne 1$ since $H$  is proper in $P_i$. As $P_i$ is abelian, $P_i
\cap T\unlhd P_iT=G$. But $P_i$ is minimal normal in $G$, so $P_i\cap T=P_i$ and thereby $T=G$. Thus
 $G$ is the only supplement of $H$ in $G$.
 If $H$ has a supersolvable supplement in $G$, then $G$ is supersolvable
 and $|P_i|=p$, which contradicts that $|D|<|P_i|$. If $H$ is weakly s-supplement
 in $G$ then $H=H\cap G$ is s-permutable in $G$. By Lemma \ref{spn}, $O^p(G)\le N_G(H)$.
It follows that $H\unlhd G=G_pO^p(G)$. This is nonsense for $P_i$ is minimal normal in $G$. Thus $|D|\ge |P_i|$ for any $i$.

If  $|P_i|=|D|$ for any $i$, then we can see that the conclusion holds. Assume $|P_i|<|D|$ for some $i$.
 Without loss of generality, we can assume that $i=1$. Clearly $P/P_1=P_2P_1/P_1\times\cdots\times P_rP_1/P_1$
 and $P/P_1\cap \Phi(G/P_1)=1$ by \cite[A,(9.11)]{D-H-BOOK}. By Lemma \ref{sk1}, one can verify that the
 hypotheses still hold on $G/P_1$. Thus  $|P_2P_1/P_1|=\cdots=|P_rP_1/P_1|$ and $|D|/|P_1|=|P_2P_1/P_1|^{m_1}$ for some positive integer $m_1$.
 It follows that $|P_2|=\cdots=|P_r|$ and $|D|/|P_1|=|P_2|^{m_1}$.
 In particularly, $|P_2|<|D|$ and the hypotheses also hold on
 $G/P_2$.  If $r\ge 3$, then $|P_1|=|P_3|=\cdots=|P_r|$. It follows that $|P_1|=|P_2|=|P_3|=\cdots=|P_r|$
 and $|D|=|P_1|^{m_1+1}=|P_1|^m$ where $m=m_1+1$. If $r=2$ then  $P/P_1\cong P_1P_2/P_1$ is minimal normal in $G/P_1$.
   Since the hypotheses holds on $G/ P_1$, we can find a contradiction as above.  Thus the lemma holds.

\end{proof}

\begin{corollary}\label{c32}
    Let $P$ be a normal $p$-subgroup of $G$ with $P\cap \Phi(G)=1$. Assume that $D$ is a subgroup of $P$ with $1<D<P$ and  every subgroup of order $|D|$
    of $P$ having no supersolvable supplement in $G$ is weakly s-supplement in $G$. If $(\iota(P),\iota (D))=1$
     or $(\iota(|P:D|),\iota (D))=1$ then $P\subseteq Z_\infty^{\frak U}(G)$.
\end{corollary}
\begin{proof}
   By Lemma \ref{31}, $P=P_1\times P_2\times\cdots\times P_r$, where $P_1,  P_2, \cdots, P_r$
   are all minimal normal in $G$, and $\iota(P_i)$, $i=1, \cdots, r$, is
   a common divisor of $\iota(P),\iota (D)$ and $\iota(|P:D|)$. Thus if $(\iota(P),\iota (D))=1$
     or $(\iota(|P:D|),\iota (D))=1$ then $P\subseteq Z_\infty^{\frak U}(G)$.
\end{proof}

\begin{lemma}\label{32}
        Let $P$ be a normal $p$-subgroup of $G$  and $D$   a subgroup of $P$ with $1<D<P$. Assume that   every subgroup of order $|D|$ or $2|D|$
        (when $P$ is a nonabelian 2-group) of $P$ having no supersolvable supplement in $G$ is weakly s-supplement in $G$. If
        $P'< P\cap \Phi(G)$ or $|D|\le |P'|$, then $P\subseteq Z_\infty^{\frak U}(G)$.
\end{lemma}

\begin{proof} Assume that the lemma does not hold and choose $P$ be a
counter example with minimal order. We prove the lemma via the following steps.

(1) {\it $P\nsubseteq \Phi(G)$ and $P/P\cap \Phi(G)=P_1/P\cap \Phi(G)\times P_2/P\cap \Phi(G)\times \cdots P_r/P\cap \Phi(G)$,
 where $P_i/P\cap \Phi(G),\ i=1,2,\cdots, r$,
  is minimal normal in $G/P\cap \Phi(G)$}.

If $P\subseteq \Phi(G)$, then for any subgroup $H$ of $P$, $G$ is the only supplement of $H$ in $G$. If there is a subgroup $H$
of order $|D|$ (or 2$|D|$ when $P$ is an nonabelian 2-group) has a supersolvable supplement in $G$ then $G$ is supersolvable and $P \subseteq Z_\infty^{\frak U}(G)$. Now assume
that every such subgroup $H$ has no supersolvable supplement in $G$. Then $H$ is weakly s-supplement in $G$ and hence $H=H\cap G$
is s-permutable in $G$. It follows from \cite[Lemma 3.1]{On-Pi-property} that $P \subseteq Z_\infty^{\frak U}(G)$.
This contradiction shows that $P\nsubseteq \Phi(G)$.
 Clearly, $(P/P\cap \Phi(G))\cap \Phi(G/P\cap \Phi(G))=1$. Hence by Lemma \ref{lesoc},
 $P/P\cap \Phi(G)=P_1/P\cap \Phi(G)\times P_2/P\cap \Phi(G)\times \cdots P_r/P\cap \Phi(G)$,
 where $P_i/P\cap \Phi(G),\ i=1,2,\cdots, r$,
  is minimal normal in $G/P\cap \Phi(G)$.

(2) $|D|>p$.

Assume that $|D|=p$. Then the hypotheses hold on $P_i$ for any $i\in \{1,\cdots,r\}$. If $r>1$,
then $P_i\subseteq Z_\infty^{\frak U}(G)$ for any $i$ and so $P\subseteq Z_\infty^{\frak U}(G)$.
Assume that $r=1$. Then $P/P\cap \Phi(G)$ is a $G$-chief factor. If $P\cap\Phi(G)=p$, then
$P\cap \Phi(G)\subseteq Z_\infty^{\frak U}(G)$
clearly holds. If $P\cap\Phi(G)\ne 1$
then the hypotheses  hold on $P\cap \Phi(G)$ and hence $P\cap \Phi(G)\subseteq Z_\infty^{\frak U}(G)$.
Thus $P\cap \Phi(G)\subseteq Z_\infty^{\frak U}(G)$
always holds. Let $H$ be any cyclic subgroup of order
$p$ or 4 (if $P$ is a nonabelian 2-group). If $H(P\cap \Phi(G))=P$
then $P/P\cap \Phi(G)\cong H/H\cap \Phi(G)$ is cyclic. Since $P\cap
\Phi(G)\subseteq Z_\infty^{\frak U}(G)$, we see that $P\subseteq
Z_\infty^{\frak U}(G)$. Assume that $H(P\cap \Phi(G))<P$ for any such
subgroup $H$. Let $T$ be any supplement of $H$ in $G$. We claim that $T=G$. If $T<G$, then $T (P\cap
\Phi(G))<G$. Since $PT=HT=G$ and $P/P\cap \Phi(G)$ is an abelian minimal normal subgroup of $G/P\cap\Phi(G)$,
$T (P\cap\Phi(G))/P\cap\Phi(G)$ is a complement of $P/P\cap\Phi(G)$. But $G/(P\cap\Phi(G))=HT/(P\cap\Phi(G))$,
so $P=H(P\cap\Phi(G))$, a contradiction.
Thus our claim holds. Consequently, $G$ is the only supplement of
$H$ in $G$. Hence $H$ is s-permutable in $G$ if $H$ has no
supersolvable supplement in $G$ by the hypotheses. By \cite[Lemma
3.1]{On-Pi-property}, $P \subseteq Z_\infty^{\frak U}(G)$, a contradiction. Hence (2) holds.

 (3) {\it $|P'|<|D|$.}

   Assume that $|D|\le
    |P'|$.  Since $P'\le P\cap \Phi(G)<P_i$, $|D|<|P_i|$, $i=1,2,\cdots, r$.
    To prove $P\subseteq Z_\infty^{\frak U}(G)$, it is sufficient to prove  that
     $P_i\subseteq Z_\infty^{\frak U}(G)$ for every $i$.
   If $P_i'=P'$, then $|D|<|P_i'|$. If $P_i'<P'$, then
    $P_i'<P_i\cap \Phi(G)$ since $P'\le\Phi(P)\le\Phi(G)$.  Therefore,  the hypotheses still hold on $(G, P_i)$.
    If $P_i<P$, then by induction on $|P|$, we have that $P_i \subseteq Z_\infty^{\frak U}(G)$.
    Now, assume that $P_i=P$.
    Then $P/P\cap \Phi(G)$ is a chief factor of $G$.

  Suppose  that $P'=P\cap \Phi(G)$. Then $P/P'$ is a chief factor of $G$. Let $H$ be a subgroup of order
   $|D|$ (or 2$|D|$ when $P$ is an nonabelian 2-group) of $P$.
   We claim that $G$ is the only supplement of $H$ in $G$.  Clearly,
     if $H\le P'\le \Phi(G)$, then $HT=G$ if and only if $T=G$. Assume that
     $H\nsubseteq P'$ and $T$ is a supplement
     of $H$ in $G$. If $T<G$, then $TP'$ is still a proper subgroup of $G$. Since
     $G=HT=PT$, we have that $(P/P')
     (TP'/P')=G/ P'$. As $P/P'$ is minimal normal
     in $G/P\cap \Phi(G)$ and
     $TP'/P'$ is proper in $G/P'$,
      $(P/P')\cap( TP'/P')=1$ and
      $TP'/P'$ is a complement of $P/P'$ in $G/P'$. But $G=HT$, so $TP'/P'$ is a
      supplement of $HP'/P'$ in $G/P'$. This induces that $HP'/P'=P/P'$ and
      so $H=P$, a contradiction. Thus our claim holds. It follows that all subgroups of order $|D|$
       (and 2$|D|$ when $P$ is an nonabelian 2-group) in $P$
      having no supersolvable supplement in $G$ are s-permutable in
      $G$ and hence $P \subseteq Z_\infty^{\frak U}(G)$ by \cite[Lemma 3.1]{On-Pi-property}.

Assume that $P'<P\cap \Phi(G)$. Then by induction on $P$, $P\cap
\Phi(G)\subseteq Z_\infty^{\frak U}(G)$. Let $N$ be a minimal normal
subgroup of $G$ contained in $P'< P\cap \Phi(G)$. Then $N$ is of order
$p$. By (2), $|D|>p$ and so the hypotheses still hold on $(G/N, P/N)$ by Lemma \ref{sk1}.
 Hence $P/N\subseteq Z_\infty^{\frak U}(G/N)$.
It follows directly from $|N|=p$ that $P\subseteq Z_\infty^{\frak
U}(G)$.  This contradicts the choice of $P$ and hence  $|P'|<|D|$.

(4) {\it $P'=1$ and $P$ is abelian}.

 Since $|D| \nleqslant |P'|$ by (3),  $P'<P\cap \Phi(G)$ by the hypotheses. Then
  it can be  verified that the hypotheses hold on $(G/P', P/P')$ by Lemma
     \ref{sk1}. If $P'\ne 1$, then $P/P'\subseteq Z_\infty^{\frak U}(G/P')$ by induction. Since
     $P'\le \Phi(P)$, It follows from Lemma \ref{2.7} that
     $P\subseteq Z_\infty^{\frak U}(G)$, which contradicts the choice of $P$. Thus (4) holds.

(5) {\it Let $N$ be a minimal normal subgroup of $G$ contained in $P\cap\Phi(G)$. Then $|N|<|D|$}

If $|D|< |N|$, then the hypotheses holds on $(G, N)$ and
hence $N\subseteq Z_\infty^{\frak U}(G)$ by (1).
It follows that $|N|=p$ and $|N|\le |D|$, a contradiction.

    Suppose that $|N|=|D|$. In this case, we claim that $N$ is cyclic. Let $L/N$ be a chief factor of $G_p$, where $G_p$
    is a Sylow $p$-subgroup of $G$.  Then $|L|=p|N|=p|D|$. Let $\mho=\langle x^p\mid x\in l\rangle$. Then $\mho\le\Phi(L)\le N$. If $\mho=N$,
then $L$ is cyclic since $L/\Phi(L)=L/N$ is cyclic. It follows that $N$ is cyclic. Assume that $\mho<N$.
 Clearly, $\mho\unlhd G_p$ and hence there is a maximal subgroup $M$ of $N$ such that $\mho\le M$ and $M\unlhd G_p$.
Choose an element $x\in L\setminus N$.
Then $x^p\in \mho\le M$ and  $H=\langle M,\ x\rangle$ is of order $p|M|=|N|=|D|$.   Let $T$ be
any supplement of $H$ in $G$. Since $P$ is abelian,
    $N\cap T\unlhd PT=HT=G$. If  $N\nsubseteq T$ then $N\cap T=1$ by the minimality of $N$. Thus $|NT|=|N||T|\ge|HT|=G$ and thereby, $G=NT$,
    which is contrary to $N\subseteq \Phi(G)$. Hence $N\subseteq T$.
    Assume that $H$ has a Supersolvable supplement $T$ in $G$. Then there is a
    cyclic subgroup $R$ of $N$ such that $R\unlhd T$. Since $P$ is abelian,
    $R\unlhd PT=HT=G$. By the minimality of $N$, we have that $N=R$ is cyclic. Assume that $H$ has no  Supersolvable supplement in $G$.
    Then $H$ is weakly s-supplement in $G$ by hypotheses.
     It follows that $M=H\cap T\cap N$ is s-permutable in $G$. Since $M\unlhd G_p$, $M\unlhd G=G_pO^p(G)$. Again by the minimality
     of $N$, we have that $M=1$ and so $N$ is cyclic. Thus our claim holds. This implies that $|N|=|D|=p$. But $|D|>p$ by (2), a contradiction.
     Hence $|N|<|D|$ and (5) holds.

    By the hypotheses, we can see that $P\cap \Phi(G)\ne 1$.
    In the following, $N$ denotes always  a minimal normal subgroup of $G$ contained in $P\cap\Phi(G)$.

(6) {\it $N$ is the unique minimal normal subgroup of $G$ contained in $P$.}

    Assume that this does not holds and  $P$ contains a minimal normal subgroup $L$ of $G$ different from $N$.
    Since $|N|<|D|$ by (5), by Lemma \ref{sk1}, every subgroup of order $|D|/|N|$ of $P/N$ having no supersolvable supplement in $G/N$ is
    weakly s-supplemented in $G/N$.  If $P/N\cap \Phi(G/N)\ne 1$ then the hypotheses still hold on $(G/N, P/N)$ and so
    $P/N\subseteq Z_\infty^{\frak U}(G/N)$. Hence $L$ is of order $p$
     and $|L|<|D|$ since $|D|>p$ by (2).
    If $P/N\cap \Phi(G/N)= 1$,  then  by Lemma \ref{31}, $P/N=P_1/N\times
    P_2/N\times\cdots\times P_r/N$, where $P_1/N,  P_2/N, \cdots, P_r/N$ are all minimal normal in $G/N$, $|P_1/N|=|P_2/N|=\cdots=|P_r/N|$ and
    $|D|/|N|=|P_1/N|^m$ for some positive integer $m$.
    In particularly, $|L|=|P_1/N|<|D|$. Thus $|L|<|D|$ holds in both case and hence
     the hypotheses hold on $(G/L, P/L)$ since, clearly, $NL/L$ is contained in $\Phi(G/L)$. Thus $P/L\subseteq
     Z_\infty^{\frak U}(G/L)$. In particularly, $N\cong NL/L$ is cyclic. If $P/N\subseteq Z_\infty^{\frak U}(G/N)$, then $P\subseteq
     Z_\infty^{\frak U}(G)$. Thus $P/N\cap \Phi(G/N)= 1$.
     Also, if $L$ is cyclic, then $P\subseteq Z_\infty^{\frak U}(G)$ since $P/L\subseteq Z_\infty^{\frak U}(G/L)$.
     Assume that $L$ is noncyclic and $P/N=P_1/N\times P_2/N\times\cdots\times P_r/N$, where $P_1/N,  P_2/N, \cdots, P_r/N$ are all minimal
     normal in $G/N$ and $|P_1/N|=|P_2/N|=\cdots=|P_r/N|$. Since $P/L\subseteq Z_\infty^{\frak U}(G/L)$,
     any $G$-chief factor between $L$ and $P$ is cyclic. Thus if $r\ne 1$, then by \cite[Theorem 1.6.8]{GUO-BOOK}, there is some $i$ such that $P_i/N\cong L$ is noncyclic and $P_j/N\cong Q/L$ is cyclic for any $j\ne i$,
     where $Q/L$ is a $G$-chief factor contained in $P$.
     But $|P_i/N|=|P_j/N|$, a contradiction. Hence $r=1$. This induces that
     $P/N$ is minimal normal in $G/N$ and  $|P/N|=|P_1/N|<|D|/|N|$, a contradiction. Thus (6) holds.

   (7) $\Phi(P)= 1$.

   Assume that $\Phi(P)\ne 1$. Then  $N\subseteq \Phi(P)$ by (6).
    If $P/N\subseteq Z_\infty^{\frak U}(G/N)$, then $P\subseteq
    Z_\infty^{\frak U}(G)$ by
    Lemma \ref{2.7}. Suppose that $P/N\nsubseteq Z_\infty^{\frak U}(G/N)$. Then $\Phi(G/N)\cap P/N=1$
    and $P/N=P_1/N\times P_2/N\times\cdots\times P_r/N$ by Lemma \ref{31}. Moreover, $|P_1/N|=|P_2/N|=\cdots=|P_r/N|$ and $|D|/|N|=
   |P_1/N|^m$ for some
    positive integer $m$. Assume that exp$P=p$. Then $\Phi(P)=\mho_1(P)=1$, where $\mho_1(P)=\{a^p\mid a\in P\}<P$.  This contradicts $\Phi(P)\ne 1$.
    Hence exp$P> p^2$. Then exp$(P/N)> p$ and so $\Phi(P/N)\ne 1$. This contradicts that $\Phi(G/N)\cap P/N=1$. Thus  exp$P=p^2$. If $D$
    is maximal in $P$, then $p=|P:D|=|P/N:D/N|=\frac{|P_1/N|^r}{|P_1/N|^m}$. Hence $r=m+1$ and $|P_1/N|=|P_2/N|=\cdots=|P_r/N|=p$. This induces
    that $P/N\subseteq Z_\infty^{\frak U}(G/N)$, a contradiction. Thus $|P:D|\ge p^2$. We claim that $N$ is cyclic. Otherwise, there must be a
    subgroup $H$ of order $|D|$ such that $H\cap N\ne 1$ since exp$P=p^2$ and $|P:D|\ge p^2$.  Moreover, we can choose that $H\cap N\unlhd G_p$,
     where $G_p$ is some Sylow $p$-subgroup of $G$.  Let $T$ be any supplement of $H$ in $G$. Then $P\cap T\unlhd PT=HT=G$. Clearly $P\cap T\ne 1$ so
     $N\subseteq T$ as $N$ is the only minimal normal subgroup of $G$
     contained in $P$. If $T$ is supersolvable, then $N$ has a cyclic
     subgroup $R$, which is
     normal in $T$. But $P$ is abelian, so $R$ is normal in $G$ and hence $N=R$ is cyclic.
      Assume that $H$ is weakly s-supplemented in $G$. Then $N\cap H\le T
     \cap H\le H_{sG}$ and so $H\cap N=H_{sG}\cap N$ is s-permutable in $G$. This induces that $O^p(G)\subseteq N_G(N\cap H)$ and so $N\cap H \unlhd G$, which contradicts the minimality of $N$. Thus $N$ is cyclic and our claim holds.

    Since $\Phi((G/N)\cap (P/N))=1$, we have that $N=\Phi(P)$. Let $P=\langle a_1\rangle\times\cdots\langle a_k\rangle\times\langle a_{k+1}\rangle\times\cdots
     \times\langle a_n\rangle$,  where $|a_1|=\cdots=|a_k|=p^2$ and $|a_{k+1}|=\cdots=|a_n|=p$. Then $k=1$ since $N=\Phi(P)$ is of order $p$. It follows
     that $\Omega=\Omega_1(P)=\{a\mid a^p=1\}$ is a maximal subgroup of $P$. Clearly, $\Omega\cap \Phi(G)=P\cap
     \Phi(G)\ne 1$. Note that $D$ is not maximal in $P$. Hence
      $|D|<|\Omega|$ and so $\Omega \subseteq
     Z_\infty^{\frak U}(G)$ by induction. But $P/\Omega$ is of order $p$, so $P\subseteq Z_\infty^{\frak U}(G)$. This contradiction shows that (7)
     holds.

     {\it The  final contradiction}

     Since $\Phi(P)=1$,
    $P$ is an elementary abelian $p$-group. In particularly, $N$ is complemented in $P$. Assume that $N$ is noncyclic. Then there
     exists a subgroup $H$ of order $|D|$ such that $1<H\cap N<N$ and
      $H\cap N\unlhd G_p$, where $G_p$ is a Sylow $p$-subgroup
      of $G$. As
    above argument, one can find a contradiction. Thus $N$ is cyclic. Let $H_1/N$ be any subgroup of $P/N$ of
    order $|D|$ and $H$ be a complement of $N$ in
     $H_1$. Then $H$ is of order $|D|$. If $H_1/N$ has no supersolvable supplement in $G/N$, then
     $H$ has no supersolvable supplement in $G$ and so is weakly s-supplement in $G$ by hypotheses. Let $T$
      be any supplement of $H$ in $G$ with $H\cap T\le H_{sG}$. Since $N\le T$ by above argument, $T/N$ is a supplement of $H_1/N$ in
     $G/N$ and $(H_1/N)\cap (T/N)=(H_1\cap T)/N=(H\cap T)N/N\le H_{sG}N/N$.  Hence $H_1/N$ is weakly s-supplemented
      in $G/N$.  If $P/N\cap \Phi(G/N)\ne 1$,
       then the hypotheses hold on $(G/N, P/N)$ and so
       $P/N\subseteq Z_\infty^{\frak U}(G/N)$. It follows from $|N|=p$ that $P\subseteq Z_\infty^{\frak U}(G)$.
        If $P/N\cap\Phi(G/N)=1$, then $P/N=Q_1/N\times Q_2/N\times\cdots\times Q_s/N$,
       where $Q_i/N$ is minimal normal in $G$, $i=1,
      \cdots, s$,  $|Q_1/N|=|Q_2/N|=\cdots=|Q_s/N|$ and $|D|=|Q_1/N|^{m'}$ for some integer $m'$ by Lemma \ref{31}. But we have that $P/N=P_1/N\times
      P_2/N\times\cdots\times P_r/N$, so $|Q_1/N|=|P_1/N|$ by \cite[Theorem 1.6.8]{GUO-BOOK}. Thus $|P_1/N|^m=|D|/p=|P_1/N|^{m'}/p$ for some integers
      $m$ and $m'$. This implies  that $m'=m+1$ and $|P_1/N|=p$. Therefore $P/N\subseteq Z_\infty^{\frak U}(G/N)$ and then, $P\subseteq Z_\infty^{\frak
      U}(G)$ is cyclic since $N$ is cyclic. The final contradiction  completes the proof.

\end{proof}

\begin{lemma} \label{35}
    Let $G$ be a  group and $p$ the minimal prime divisor of the order of $G$. Assume that $P$ is a Sylow $p$-subgroup of $G$ and $D$ is a
    nontrivial proper  subgroup of $P$. If every  subgroup of $P$ of order $|D|$ or $2|D|$ (when $P$ is a nonabelian 2-group) having no
    supersolvable supplement in $G$ is weakly s-supplemented in $G$, then $G$ is $p$-solvable and the $p$-length of $G$ is 1.
\end{lemma}
\begin{proof}
Assume the lemma does not holds and let $G$ be a counter example of minimal order. Then $G$ is not $p$-nilpotent. We proceed the proof via
 the following steps.

(1) $O_{p'}(G)=1$

By Lemma \ref{sk1}, it can be verified that the hypotheses still hold on $G/O_{p'}(G)$ and if $O_{p'}(G)\ne 1$   then $G/O_{p'}(G)$ is $p$-solvable
and the $p$-length is 1. It follows  that $G$ is $p$-solvable and the $p$-length of $G$ is 1. So we can assume that $O_{p'}(G)=1$.

(2) {\it Let $N$ be a minimal normal subgroup of $G$. If $N$ is a $p$-group, then $|N|=|D|$.}

If $|D|<|N|$, then $N$ has a proper subgroup $H$ of order $|D|$. Since $N$ is minimal normal in $G$ and $N$ is abelian, $G$ is the only supplement of
$H$ in $G$. If $H$ has a supersolvable supplement in $G$, then
 $G$ is supersolvable and so the $p$-length of $G$ 1, which
contradicts the choice of $G$. Assume that every such subgroup $H$ is weakly s-supplemented in $G$. Then  $H$ is s-permutable in
$G$. Without loss of generality, we can assume that $H$ is normal
 in $P$. Then $H\unlhd \langle P, O^p(G)\rangle=G$,
 which contradicts the minimality of $N$.
Thus $|N|\le |D|$.

If $|N|<|D|$, then the hypotheses still hold on $G/N$. Therefore, $G/N$
 is $p$-solvable and the $p$-length is 1. It follows that $G$ is
$p$-solvable. Assume that $G$ has
another minimal normal subgroup $L$. Then, similarly,   $G/L$ is also
$p$-solvable and its $p$-length is 1.
This induces that $G\cong G/N\cap L$ is  $p$-solvable and the
$p$-length of it is 1, a contradiction. Thus $N$ is the unique
minimal normal
 subgroup of $G$. Since the class of all $p$-solvable groups with the $p$-length is 1 is a saturated formation, we have that $N\nsubseteq \Phi(G)$.
 It follows that $G=N\rtimes M$ for some maximal subgroup $M$ of $G$ and $N=O_p(G)$. Clearly, $O_{p'}(M)\ne 1$ and $|P\cap M|>|D|/|N|$.
  Let $P_1$ be an subgroup of $P\cap M$ of order $p|D|/|N|$and $M_1=P_1 O_{p'}(M)N$. Then $P_1N$ is a Sylow $p$-subgroup of $M_1$,
  and every maximal subgroup $H$ of $P_1N$ is of order $|D|$. Thus, if $H$ has no supersolvable supplement in $M_1$, then $H$ is
  weakly s-supplemented in $M_1$ by Lemma \ref{sk1}, and then by Lemma \ref{max}, $M_1$ is $p$-nilpotent. But $O_p(G)=N\le M_1$, so $O_{p'}(M_1)=1$ by
   Lemma \ref{OP}. Thus $M_1$ is a $p$-group. This contradiction shows that (2) holds.

  (3) $O_p(G)=1$

   Assume that $O_p(G)\ne 1$. Suppose $\Phi(G)\cap O_p(G)\ne 1$ and let $N$ be a minimal normal subgroup of $G$ contained in $\Phi(G)\cap O_p(G)$.
   By (2), $|N|=|D|$. If $N<O_p(G)$ then $O_p(G)\subseteq Z_\infty^{\frak U}(G)$ by Lemma \ref{32}. It follows that $O_p(G)\subseteq Z_\infty(G)$ since
   $p$ is the minimal prime divisor of $|G|$. Hence $N$ is cyclic and $|N|=|D|=p$. It follows from Lemma \ref{minimal} that $G$ is $p$-nilpotent, a contradiction. Assume $N=O_p(G)$.
  Let $L/N$ be a minimal normal subgroup of $P/N$. Then $|L|=p|N|=p|D|$ and $N$ is maximal in $L$. It follows that $\mho_1(L)=
  \langle x^p\mid x\in L\rangle\le \Phi(L)\le N$. If $\mho_1(L)=N$, then $N=\Phi(L)$ and hence $L$ is cyclic. This induces
   that $N$ is cyclic and so $|D|=|N|=p$. In this case, by Lemma \ref{minimal}, $G$ is $p$-nilpotent, a contradiction. Therefore, $\mho_1(L)<N$.
    Let $M$ be a maximal subgroup of $N$ such that $M\unlhd P$ and $\mho_1(L)\le M$. Choose $x\in L\setminus N$. Then $H=\langle M, x\rangle$ is a
    subgroup of order $|N|=|D|$. Let $T$ be a supplement of $H$ in $G$. Suppose $T<G$. Since $N\le \Phi(G)$,
     $TN<G$. But $|G:TN|=|TL:TN|=p\frac{|T\cap N|}{|T\cap L|}\le p$, so $TN$ is maximal in $G$ and $|G:NT|=p$.
   Thus $NT$ is normal in $G$ by the minimality of $p$.   If $N$ is not a Sylow $p$ subgroup of $NT$, then the hypotheses still hold on $NT$ and hence
   $NT$ is $p$-nilpotent. Since $|G:NT|=p$ and $NT\unlhd G$, $G$ is $p$-nilpotent. If $N$ is a Sylow $p$ subgroup of $NT$,
   then $D$ is maximal
   in $P$ and so $G$ is $p$-nilpotent by Lemma \ref{max}. This  contradiction shows that $G$ is the only supplement of $H$ in $G$. If $T$ is supersolvable, then
   $G=T$ is $p$-nilpotent, a contradiction.
   Assume that $H$ has no supersolvable supplement in $G$. Then $H$ is weakly s-supplemented in $G$ by the hypotheses and consequently,  $H=H\cap G$ is
    s-permutable in $G$. It follows that $M=H\cap N$ is s-permutable in $G$ and so $M\unlhd PO^p(G)=G$. By the minimality of $N$,
    we have that $M=1$ and $N$ is cyclic of order $p$. Still by Lemma \ref{minimal}, $G$ is $p$-nilpotent, a contradiction.

    Suppose that $\Phi(G)\cap O_p(G)= 1$. Then $O_p(G)$ is abelian. Let $N$ be a minimal normal subgroup of $G$
     contained in $O_p(G)$. Then $N$ is of order $|D|$ by (2) and is complemented in $G$ by $\Phi(G)\cap O_p(G)= 1$.
    Let $G=N\rtimes M$. Clearly, $O_p(M)\subseteq O_p(G)$. If $O_p(M)\ne 1$, then  $O_p(M)\unlhd MO_p(G)=G$. Let $L$
     be a minimal normal subgroup of $G$ contained in $O_p(M)$.
     Then the order of $L$  is  $|D|$ by (2). If $L$ is a Sylow $p$-subgroup of $M$,
     then $O_p(G)=NL$ is a Sylow $p$-subgroup of $G$. This is contrary to the choice of $G$.
      Thus the order of a Sylow $p$-subgroup of $M$ is greater than $|D|$
     and so, the hypotheses hold on $M$.  Therefore, $G/N\cong M$ is $p$-solvable and the $p$-length of it is 1. By the same
     argument, we have that
    $G/L$ is also $p$-solvable and the $p$-length of $G/L$ is 1. Therefore, $G\cong G/L\cap N$ is $p$-solvable and the $p$-length of it is 1,
    which contradicts the choice of $G$. Hence $O_p(M)=1$. Now assume
     that $O_{p'}(M)\ne 1$. Let $x\in P\cap M$ of order $p$ and $P_1=\langle N, x\rangle$.
    Then $|P_1|=p|D|$.
    Since $NO_{p'}(M)\unlhd G$, $X=O_{p'}(M)P_1=O_{p'}(G)NP_1$ is a subgroup of $G$ and  every maximal subgroup of
    $P_1$ is of order $|D|$. Hence by Lemma \ref{max} that $X$ is $p$-nilpotent and so is $NO_{p'}(M)$. This induces that
    $O_{p'}(M)$ char $NO_{p'}(M)\unlhd G$, which contradicts $O_{p'}(G)=1$. Thus $O_{p'}(M)=O_p(M)=1$ and in particularly,
    $G$ is not solvable. If $p>2$, then $G$ is of odd
    order and so is solvable, a contradiction. Hence $p=2$. Let $R$ be a minimal subnormal subgroup of $M$. Then $R$ is nonabelian
    and  $p=2$ is a divisor of $|R|$.
     Let $G_1=NR$. Then the hypotheses still hold on $G_1$ since $|N|=|D|<|G_{1_{p}}|$,
     where $G_{1_p}$ is a Sylow $p$-subgroup of $G_1$. If $G_1<G$, then $G_1$ is $p$-solvable and so is $R$, a contradiction. Hence
    $G=G_1=NR$. If the Sylow $p$-subgroup $P$ is abelian, then $P\cap R\unlhd P$ and so $(P\cap R)^G=(P\cap R)^R\le R$. Since $R$
    is simple, we have that $R=(P\cap R)^G\unlhd G$ and so $G=N\times R$. But $N$ is minimal normal in $G$, so $N$ is of order $p$. Thus $G$ is $p$-nilpotent
    by Lemma \ref{minimal} and $|N|=|D|$. Assume that $P$ is nonabelian. Then every subgroup $H$ of order $2|D|=2|N|$ having no supersolvable supplement
    in $G$ is weakly s-supplemented in $G$. By Lemma \ref{sk1}, it can be verified that every subgroup of order 2 of $G/N$
    having no supersolvable supplement  in $G/N$ is weakly s-supplemented in $G/N$. Since $R\cong RN/N=G/N$, every subgroup of order 2 of $R$
    having no supersolvable supplement  in $R$ is weakly s-supplemented in $R$. Let $H$ be a subgroup of $R$ of order 2 and $T$   a
     supplement of $H$ in $R$. If $T\ne R$, then $|R:T|=2$ and so $T\unlhd R$, which contradicts that $R$ is simple. Hence $T=R$.
     If $R$ supersolvable then $G$ is solvable, a contradiction. Hence $H$ is weakly s-supplemented in $R$. But $R$ is the only supplement
     of $H$ in $R$, so $H=H\cap R$ is s-permutable in $R$ and hence $O_p(R)\ne 1$, which contradicts that $R$ is simple.
     This contradiction shows that (3) holds.

(4) {\it $G$ is simple.}

Let $N$ be a minimal normal subgroup of $G$. By (1) and (3), $N$ is a nonabelian $pd$-group. If $|P\cap N|\le |D|$, then there is a subgroup $P_1$ of
$P$ with $N\cap P<P_1$ and $|P_1|=p|D|$. Let $X=NP_1$. Then every maximal subgroup of the Sylow $p$-subgroup $P\cap X$
of $X$ is of order $|D|$. By Lemma \ref{max}, $X$ is $p$-nilpotent and so is $N$, a contradiction.
If $|P\cap N|>|D|$, then the hypotheses hold on $N$. If $N<G$, then $N$ is $p$-solvable by the choice of $G$, a contradiction. Hence $N=G$
 and $G$ is simple.

 {\it The final contradiction }

 If $p>2$, then $G$ is solvable  and so is abelian, a contradiction. Hence $p=2$. Assume that $P$ is abelian and
 $H$ is a subgroup of $P$ of order $|D|$. If $H$ is weakly s-supplemented in $G$ and let $T$ be a supplement of $H$ in $G$ with $H\cap T
 \le H_{sG}$. Then $T\ne G$ since $H$ could not s-permutable in $G$. Clearly, $P\cap T\ne 1$
 since $D<P$.  Hence $1\ne (P\cap T)^G=(P\cap T)^{PT}=(P\cap T)^T
 \le T<G$, which contradicts (4).  Now consider that $P$ is nonabelian. Then every subgroup of order $|D|$
 or $2|D|$ having no supersolvable supplement in $G$ is weakly s-supplemented in $G$. By Lemma \ref{max},
 $D$ is not maximal in $P$ and so $2|D|<|P|$. If all subgroups of order $|D|$(or of order $2|D|$) have
 supersolvable supplements in $G$, then all maximal subgroups of $P$ have supersolvable supplements in $G$. By Lemma \ref{max}, $G$ is $p$-nilpotent, a contradiction. Hence there is a subgroup
 $H_1$ of order $|D|$ and a subgroup $H_2$ of order $2|D|$ are weakly s-supplemented in $G$. But $G$ is simple, so both $H_1$ and $H_2$ are supplement
 in $G$. Hence there are subgroups $T_1$ and $T_2$ with $|G:T_1|=|D|$ and $|G:T_2|=2|D|$. In view Lemma \ref{prime-index}, such a nonabelian
 simple group does not exist and our lemma holds.

\end{proof}

\noindent{\bf Proof of Theorem \ref{B}}.
We first prove that $E$ satisfies  Sylow tower property (see \cite[p5]{Weinstein}). In fact,
by Lemma \ref{sk1}, it can be verified that the hypotheses still holds on $E$. If $E<G$ then $E\in \frak{U}$ by induction and hence $E$satisfies  Sylow tower property in this case.
Now assume that $E=G$. Let $p$ be the minimal prime divisor  of $|G|$ and $P$ a Sylow $p$-subgroup of $G$. It follows from Lemma \ref{35} that
 $G$ is $p$-solvable and the $p$-length of $G$ is 1. Thus $G/O_{p'}(G)$ is $p$-closed. By Lemma \ref{sk1},  every  subgroup of $G/O_{p'}(G)$ of
 order $|D|$ or $2|D|$ (when $PO_{p'}(G)/O_{p'}(G)\cong P$ is a nonabelian 2-group) having no supersolvable supplement in $G/O_{p'}(G)$ is
  weakly s-supplemented in $G/O_{p'}(G)$.
  If $\Phi(P)\ne P'$, then $P'O_{p'}(G)/O_{p'}(G)<\Phi(P)O_{p'}(G)/O_{p'}(G)
   \le \Phi(G/O_{p'}(G)$. It follows from Lemma \ref{32} that $PO_{p'}(G)/O_{p'}(G)
  \subseteq Z_\infty^{\frak U}(G/O_{p'}(G))$.
   Assume that $P'=\Phi(P)$. If $|D|\le |\Phi(P)|$ then
  $|D|\le |P'|$. Again by Lemma \ref{32}, it holds that $PO_{p'}(G)/O_{p'}(G)
  \subseteq Z_\infty^{\frak U}(G/O_{p'}(G))$.  Assume that $|D|> |\Phi(P)|$. Then (iii) holds
  on $P$ and so
 $(\iota(P/P'),\iota(|D|/|P'|))=1$ or $(\iota(P),\iota(|P:D|))=1$.  By Corollary \ref{c32}, it can be verified  that $PO_{p'}(G)/P'O_{p'}(G)
  \subseteq Z_\infty^{\frak U}(G/P'O_{p'}(G))$, and it follows from Lemma \ref{2.7} that $PO_{p'}(G)/O_{p'}(G)
  \subseteq Z_\infty^{\frak U}(G/O_{p'}(G))$.   Since $p$ is the minimal prime divisor of $G$, we have that $G/O_{p'}(G)$
  is $p$-nilpotent and hence $G$ is $p$-closed. Since the hypotheses still hold on $O_{p'}(G)$, we have that $O_{p'}(G)$  satisfies  Sylow tower property,
  and consequently $G$  satisfies  Sylow tower property.

     Let $q$ be the maximal prime divisor of $|E|$ and $Q$ a Sylow $q$-subgroup of $E$. Then $Q$
 char $E\unlhd G$ and as above argument, $Q\subseteq Z_\infty^{\frak U}(G)$ since (i) or (ii) or (iii) holds on $Q$. We can also see that the hypotheses
 still holds on $G/Q$. Hence $G/Q\in \frak{F}$ by induction on the order of $G$. Since $Q\subseteq Z_\infty^{\frak U}
 (G)$ and $\frak {U}\subseteq \frak{F}$,
 we obtain that $G\in \frak{F}$. Therefore the theorem holds.

\section{Remarks, examples and some corollaries}

1. If $D$ is minimal or maximal in $P$, then $(\iota(P),\iota(D))=1$ or $(\iota(P),\iota(|P:D|))=1$ and hence  Theorems
 A and B in \cite{Guo-Xie-Li} are special cases of our results.

 2. In Lemma \ref{35} the minimality of $p$ is necessary. In fact, if $p$ is not the minimal prime divisor of the order of $|G|$,  then
 we have the following counterexample.

 \begin{example}
    Let $A=Z_3\rtimes Z_2\cong S_3$, where $Z_3$ is a cyclic subgroup of order 3, $Z_2$  a cyclic subgroup of order 2 and $S_3$
    is the symmetric group of degree 3.
    Let $B=A\wr Z_3$, the regular wreath product of $A$ by $Z_3$. Put $G=O^2(B)=\langle x\mid o(x)=3\rangle$. Then $G\cong
    (Z_3\times Z_3\times Z_3)\rtimes A_4$, where $A_4$ is the alternative  group of degree 4.
 Let $P$ be a Sylow 3-subgroup of $G$.  It can be proved that for any maximal subgroup $H$
 of $P$, $H$ is complemented in $G$.
 So every maximal subgroup  of $P$ is weakly s-supplemented in $G$. But the $p$-length of $G$ is not 1, where $p=3$.

\end{example}

3. Clearly, if a subgroup $H$ is normal, s-permutable or c-normal  in $G$, then $H$ is weakly-supplement in $G$. Hence
one can find  the following special cases of Theorem \ref{B} in the literature.

  \begin{corollary}{\rm(\cite{Buckley-per})}
  Let $G$ be a group of odd order. If all subgroups of G of prime
order are normal in $G$, then $G$ is supersolvable.
  \end{corollary}

 \begin{corollary}{\rm(\cite{Srinivasan-per})}
 If the maximal subgroups of the Sylow subgroups of $G$ are
normal in $G$, then $G$ is supersolvable.
 \end{corollary}

\begin{corollary}
{\rm(\cite{C-normal})} If all subgroups of $G$ of prime order or order 4 are c-normal
in $G$, then $G$ is supersolvable.
\end{corollary}

\begin{corollary}
{\rm(\cite{C-normal})} If the maximal subgroups of the Sylow subgroups of $G$ are c-normal
in $G$, then $G$ is supersolvable.
\end{corollary}

\begin{corollary}
{\rm(\cite{Guo-Isr})} If the maximal subgroups of the Sylow subgroups
of $G$ not having supersolvable supplement in $G$ are normal in $G$, then $G$ is supersolvable.
\end{corollary}

 \begin{corollary}
 {\rm(\cite{AL-AH})} If the maximal subgroups of the Sylow subgroups of $G$
not having supersolvable supplement in $G$ are c-normal in $G$, then $G$ is supersolvable.

 \end{corollary}

\begin{corollary}
 {\rm(\cite{Ballester-Bolinches-Wang})} Let $\mathfrak{F}$ be a saturated formation containing
$\mathfrak{U}$. If all minimal subgroups and all cyclic subgroups with order 4 of $G^\mathfrak{F}$ are c-normal in $G$,
then $G \in\mathfrak{F}$.
\end{corollary}

\begin{corollary}
{\rm(\cite{Ballester-Ped-per})} Let $\mathfrak{F}$ be a saturated formation
containing $\mathfrak{U}$ and $G$ a group with normal subgroup $E$ such that $G/E \in\mathfrak{ F}$. Assume that
a Sylow 2-subgroup of $G$ is abelian. If all minimal subgroups of $E$ are permutable in $G$, then $G \in\mathfrak{ F}$.
\end{corollary}

\begin{corollary}
{\rm(\cite{Ramadan-per})} Let $G$ be a solvable group. If all maximal subgroups of the
Sylow subgroups of $F(E)$ are normal in $G$, then $G$ is supersolvable.
\end{corollary}

\begin{corollary}
{\rm(\cite{Ballester-Ped-per})} Let $\frak{F}$ be a saturated formation
containing $\mathfrak{U}$ and $G$ a group
with a solvable normal subgroup $E$ such that $G/E \in \frak{F}$. If all
minimal subgroups and all cyclic subgroups with order 4 of $E$ are weakly s-permutable in $G$,
then $G \in \frak{F}$.
\end{corollary}

\def\cprime{$'$}

\end{document}